\documentclass[12pt]{amsart}
\usepackage{amsmath, amssymb}
\newtheorem{theorem}{\sc Theorem}[section]
\newtheorem{lemma}[theorem]{\sc Lemma}
\newtheorem{proposition}[theorem]{\sc Proposition}

\newtheorem{example}[theorem]{\sc Example}
\newcommand{\al }{\alpha }

\usepackage{abstract} % Permite maior personalização

\title[Bounded conjugacy classes]{On finite groups with bounded conjugacy classes of commutators}
\author[D. Senise]{D\'ebora Senise} 
\address{D\'ebora Senise: Department of Mathematics, University of Brasilia, Brasilia, DF, Brazil}
\email{deborasenise2502@gmail.com}

\author[P. Shumyatsky]{Pavel Shumyatsky} 
\address{Pavel Shumyatsky: Department of Mathematics, University of Brasilia, Brasilia, DF, Brazil}
\email{pavel@unb.br}

\thanks{The work of the second author was supported by FAPDF and CNPq}


\begin{document}

\maketitle

\begin{abstract}
 In 1954 B. H. Neumann discovered that if $G$ is a group in which all conjugacy classes have finite cardinality at most $m$, then the derived group $G'$ is finite of $m$-bounded order. In 2018 G. Dierings and P. Shumyatsky showed that if $|x^G| \le m$ for any commutator $x\in G$, then the second derived group $G''$ is finite and has $m$-bounded order. This paper deals with finite groups in which $|x^G|\le m$ whenever $x\in G$ is a commutator of prime power order. The main result is that $G''$ has $m$-bounded order.
\end{abstract}

\section{Introduction}

Given a group $G$ and an element $x\in G$, we write $x^G$ for the conjugacy class containing $x$. Of course, if the number of elements in $x^G$ is finite, we have $|x^G|= [G : C_G (x)]$. B. H. Neumann discovered that if all conjugacy classes in a group $G$ are finite of size at most $m$, then the order of the commutator subgroup $G'$ is finite and bounded in terms of $m$ only \cite{neumann}. A first explicit bound for the order of $G'$ was found by J. Wiegold \cite{wie}, and the best known was obtained in \cite{gura} (see also \cite{neuvl,sesha}).

There are several recent publications where Neumann's theorem has been extended in various directions (see in particular \cite{as,dms,dieshu}). Here we are interested in groups $G$ such that $|x^G|\leq m$ whenever $x$ is a commutator, that is, there are elements $a,b\in G$ for which $x=a^{-1}b^{-1}ab$. The following results are known.
\begin{theorem}\label{known}
Let $m$ be a positive integer and $G$ a group in which $|x^G|\leq m$ whenever $x$ is a commutator. Then we have
\begin{enumerate}
\item The second commutator subgroup $G''$ is finite of $m$-bounded order.
\item $G$ has a normal subgroup $H$ such that the index $[G:H]$ and the order of $\gamma_4(H)$ are both $m$-bounded.
\end{enumerate}
\end{theorem}
Throughout, we say that a quantity is $(a,b,c\dots)$-bounded to mean that it is bounded by a constant depending only on the parameters $a,b,c\dots$.

Part 1 of the above theorem was established in \cite{dieshu}. Part 2 is immediate from the main result of \cite{ebeshu}.

In the case of finite groups, Neumann's theorem admits an interesting generalisation:\medskip

\begin{proposition} \label{easy}
    Let $m$ be a positive integer and $G$ a finite group in which $|x^G|\leq m$ whenever $x$ is an element of prime power order. Then $G'$ has $m$-bounded order.
\end{proposition}

A proof of the above statement is given in the next section. In this article we examine the question whether, in the case of finite groups, Theorem \ref{known} can be generalised in a similar manner. 

More precisely, we concentrate on finite groups $G$ in which $|x^G|\leq m$ whenever $x$ is a commutator of prime power order. 
We show by means of examples that the Fitting subgroup of such a group can have arbitrarily large index. Thus, Part 2 of Theorem \ref{known} cannot be extended to the groups in question. 

The main purpose of this paper is to show that in the case of Part 1 the answer to the above question is positive. We establish the following theorem.

\begin{theorem}\label{main}
Let $m$ be a positive integer and $G$ a finite group in which $|x^G|\leq m$ whenever $x$ is a commutator of prime power order. Then $G''$ has $m$-bounded order.
\end{theorem}

In view of Theorem \ref{main} it is natural to raise the question whether a similar result holds for other commutator words. A theorem obtained in \cite{dms} says that if $w$ is a multilinear commutator word of weight $n$ and $G$ a group in which $|x^G|\leq m$ whenever $x$ is a $w$-value, then the commutator subgroup of the verbal subgroup $w(G)$ has finite $(m,n)$-bounded order. It would be interesting to see whether, given a finite group $G$, the hypothesis that $|x^G|\leq m$ for  any $w$-value $x$ of prime power order implies that the commutator subgroup of $w(G)$ has $(m,n)$-bounded order. At present we are unable to answer the question.

The notation used in this paper is mostly standard.
\section{Preliminaries}

We start this section with a proof of Proposition \ref{easy}. Throughout, we write $\pi(G)$ to denote the set of prime divisors of the order of a finite group $G$.

\begin{proof}[Proof of Propostition \ref{easy}] Recall that $G$ is a finite group in which $|x^G|\leq m$ whenever $x$ is an element of prime power order. We wish to show that $G'$ has $m$-bounded order. 

Let $p$ be a prime bigger than $m$ and assume that $p\in\pi(G)$. Let $P$ be a Sylow $p$-subgroup of $G$. Since $|x^G|\leq m$ whenever $x$ is an element of prime power order in $G$, it follows that $P$ commutes with any element of prime power order and so $P\leq Z(G)$.

Pick an arbitrary element $g\in G$ and suppose that the order of $g$ is $p_1^{\al_1}\dots p_s^{\al_s}$, where $p_i$ are distinct primes in $\pi(G)$. Write $g=g_1g_2\dots g_s$, where the element $g_i$ is of order $p_i^{\al_i}$. This can be done for example choosing $g_i$ as a generator of the Sylow $p_i$-subgroup of the cyclic group $\langle g\rangle$. In view of the above, $g_i\in Z(G)$ whenever $p_i> m$. It follows that there are less than $m$ non-central elements in the list $g_1,\dots,g_s$. Each of these non-central elements has centralizer of index at most $m$. Therefore $|g^G|\leq m^m$. This holds for an arbitrary element $g\in G$ so the claim follows from Neumann's theorem.
\end{proof}

We now provide an example of a finite group $G$ in which $|x^G|\le m$ whenever $x$ is a commutator of prime power order and the Fitting subgroup has arbitrarily large index.
\begin{example}\label{no2} {\rm Let $p_1,p_2,\dots,p_s$ be distinct odd primes, and for every $i$ let $D_i$ be the dihedral group of order $2p_i$. Let $G=\prod D_i$ be the direct product of the groups $D_i$. It is easy to see that $|x^G|\leq 2$ whenever $x$ is a commutator of prime power order. It is also clear that the Fitting subgroup of $G$ has index $2^s$. 
}
\end{example}

The above example shows that there is no obvious extension of Part 2 of Theorem \ref{known} for finite groups in which commutators of prime power order lie in conjugacy classes of size at most $m$.\\

In what follows we will require the concept of system normalizers. 
Let $G$ be a finite soluble group. Recall that $G$ has a Sylow system. This is a family of pairwise permutable Sylow $p_i$-subgroups
$P_i$ of $G$, exactly one for each prime in $\pi(G)$, and any two Sylow systems are conjugate. The system normalizer (also known as the basis normalizer) of such
a Sylow system in $G$ is $T=\cap_i N_G(P_i)$. The reader can consult \cite[pp. 235-240]{dh} or \cite[pp. 262-265]{robinson}
for general information on system normalizers. In particular, we will use the facts
that the system normalizers are conjugate, nilpotent, and $G = T\gamma_\infty(G)$ where
$\gamma_\infty(G)$ denotes the intersection of the lower central series of $G$. Moreover, the group $G$ is generated by its system normalizers. If $N$ is a normal subgroup of $G$, then the image of a system normalizer in the quotient $G/N$ is a system normalizer of $G/N$.

\begin{lemma}\label{[G,T]} Let $G$ be a finite soluble group and $T$ a system normalizer in $G$. Then $[G,T]=G'$.
\end{lemma}

\begin{proof} We just need to show that $G'\leq[G,T]$. We can pass to the quotient $G/[G,T]$ and assume that $[G,T]=1$. Then $T\leq Z(G)$. Since $G$ is generated by the conjugates of $T$, we deduce that $G=T$ is abelian. Hence, $G'\leq[G,T]$ and the lemma follows.
\end{proof} 

We will require the following theorem, which is a particular case of \cite[Theorem 1.2]{shumy}.

\begin{theorem} \label{shumyresult}
 Let $n$ be a positive integer and $G$ a group containing a subgroup $A$ such that $|[g,a]^G| \le n$ for all $g \in G$ and $a \in A$. Then the commutator subgroup of the subgroup $[G,A]$ has finite $n$-bounded order.
\end{theorem}

Suppose that $G$ is a group generated by a symmetrical set $X$ (here symmetrical means that $X=X^{-1}$). For any $g\in G$ we write $l_X(g)$ to denote the minimal number $l$ such that $g$ is a product of $l$ elements from $X$.

The next lemma is taken from \cite{dieshu}.

\begin{lemma}\label{lemm-1}
Let $H$ be a group generated by a symmetrical set $X$ and let $K$ be a subgroup of finite index $m$ in $H$. Then each coset $Kb$ contains an element $g$ such that $l_X(g)\leq m-1$.
 \end{lemma}

The following lemma is similar to Lemma 2.3 of \cite{dieshu}. For the reader's convenience we provide a proof.
\begin{lemma} \label{[G,x]} Let $m$ be a positive integer, and let $X$ be the set of all elements $g$ of a group $G$ such that $|g^G|\le m$. Assume that $G$ is generated by $X$ and let $x\in X$. Then $[G,x]$ has $m$-bounded order. 
\end{lemma}

\begin{proof}
 Since $C_G(x)$ has index at most $m$ in $G$, by Lemma \ref{lemm-1} we can choose elements $y_1, \dots , y_m $ such that $l_X(y_i) \le m-1$ and $[G, x]$ is generated by the commutators $[y_i, x]$. For each $i = 1, \dots, m$, write $y_i = y_{i1} \dots y_{i(m-1)}$ where $y_{ij} \in X$. We can write $[y_i,x]$ as a product of conjugates of $[y_{ij},x]$. Let $h_1,\dots,h_s$ be the conjugates in $G$ of elements from the set $K=\{ x, y_{ij}; 1 \le i \le m,\ 1\le j\le m-1\}$. Notice that the size of each conjugacy class containing an element of $K$ is less than or equal to $m$, and $K$ has at most $m(m-1)+1$ elements. Therefore $s$ is $m$-bounded. Let $H=\langle h_1,\dots,h_s\rangle $. It is clear that $[G,x]\le H'$, so it is sufficient to show that $H'$ has finite $m$-bounded order. Observe that $C_G(h_i)$ has index at most $m$ in $G$ for each $i=1,\dots,s$. It follows that the centre $Z(H)$ has index at most $m^s$ in $H$. Thus Schur's theorem \cite[10.1.4]{robinson} tells us that $H'$ has finite $m$-bounded order.
 \end{proof}

\section{Theorem \ref{main}}

A difficulty that one encounters while looking for an approach to a proof of Theorem \ref{main} is that the hypothesis may fail in a homomorphic image of $G$. More precisely, if $Y(G)$ denotes the set of commutators of prime power order of the finite group $G$ and if $\overline{G}$ is a homomorphic image of $G$, then in general $Y(\overline{G})$ is not necessarily the image of $Y(G)$.

To circumvent the obstacle we will work with the subset $X(G)$ of $Y(G)$ that consists of all commutators $[x,y]$ for which there exists a Sylow subgroup $P\leq G$  such that $x$ and $[x,y]$ both lie in  $P$ while $y$ can be any element of $G$. Thus, 
 $$X(G)=\{[x,y] \, |\text{ there is } P\in Syl(G) \text{ such that }  x,[x,y]\in P,  y\in G\}.$$

The main advantage of working with the set $X(G)$ rather than $Y(G)$ is that $X(G)$ survives taking quotients. More precisely, we have the following lemma.

\begin{lemma}
Let $G$ be a finite group and $N$ a normal subgroup of $G$. Set $\overline{G} = G/N$. Then $\overline{X(G)}=X(\overline{G})$. 
\end{lemma}
\begin{proof} 
    Of course, $\overline{X(G)} \subseteq X(\overline{G})$. We need to establish the inverse containment $X(\overline{G}) \subseteq \overline{X(G)}$. 
    
Pick $[\overline{x},\overline{y}] \in X(\overline{G})$. There is a Sylow $p$-subgroup $\overline{P}$ of $\overline{G}$ such that $[\overline{x}, \overline{y}] \in \overline{P}$ and $\overline{x} \in \overline{P}$. The full preimage of $\overline{P}$ is $PN$ for some Sylow $p$-subgroup $P$ of $G$. Choose $x$ and $y$ such that $\overline{x}=xN$ and $\overline{y}=yN$. Both elements $x$ and $x^y$ lie in $PN$.
    
Write $x = x'n$, where $x'\in P$ and $n\in N$.  Since $x^y\in PN$, we also have $x'^y\in PN$. Thus, there is a Sylow $p$-subgroup $P_0$ in $PN$ such that $x'^y\in P_0$. Since $P_0$ is conjugate to $P$ in $PN$, there is an element $a\in N$ such that $x'^{ya}\in P$. 

So we have 
    $$[\overline{x}, \overline{y}] = [x',y]N = [x', ya]N.$$
Since $[x', ya] \in P$ and $x' \in P$, it follows that $[x', ya]\in X(G)$ and therefore $[x,y]\in X(G)N$. This completes the proof.
\end{proof} 

Let $G$ be a finite group and $P$ a Sylow $p$-subgroup of $G$. An immediate corollary of the focal subgroup theorem \cite[Theorem 7.3.4]{go}  is that  $P\cap G'$ is generated by $P\cap X(G)$. Consequently, $G'$ is generated by $X(G)$.

\begin{lemma} \label{propriedC1} Let $G$ be a finite group such that $|x^G|\leq m$ whenever $x\in X(G)$. Let $P$ be a Sylow $p$-subgroup of $G$ for some prime $p>m$. Then $[G',P]=1$.
\end{lemma}

\begin{proof}
 For $x\in X(G)$ let $N_x$ be the normal core of $C_G(x)$, i.e., $N_x = \underset{g \in G}{\bigcap} C_G(x)^g $. As $[G:C_G(x)]\le m$, the index $[G:N_x]$ divides $m!$. 

Taking into account that $p>m$ we conclude that $P\leq N_x$. Since this happens for all $x \in X(G)$ and since $G'$ is generated by $X(G)$, it follows that $[G',P]=1$.
\end{proof}

In the next lemma we will use the well-known fact that if $A$ is a group of coprime automorphisms of a finite group $H$ such that $[H,A,A]=1$, then $A=1$ (see for example \cite[Theorem 5.3.6]{go}).
\begin{lemma} \label{soluble}
    Let $G$ be a finite soluble group such that $|x^G|\leq m$ whenever $x\in X(G)$. Then $G''$ has $m$-bounded order.
\end{lemma}
\begin{proof}
Let $\pi_1 = \{p_1, \dots , p_k \} $ the set of prime divisors of $|G|$ which are less than or equal to $m$ and $\pi_2 =\pi(G)\setminus\pi_1$. Let $H_1$ be a $\pi_1$-Hall subgroup and $H_2$ a $\pi_2$-Hall subgroup of $G$. 

By Lemma \ref{propriedC1}, $H_2$ centralizes $G'$. Therefore $G'=K_1\times K_2$, where $K_1=G'\cap H_1$ and $K_2=G'\cap H_2$. Obviously $G$$''$ has trivial intersection with $K_2$ so we pass to the quotient $G/K_2$ and simply assume that $G'\leq H_1$. In particular $H_1$ is normal and $[H_1,H_2,H_2]=1$. It follows that $G=H_1\times H_2$. Therefore $G''={H_1}''$ and so without loss of generality we can assume that $G=H_1$.

Let $\{ P_1, \dots , P_k \}$ be a Sylow basis of $G$ and $T = \bigcap_{i=1}^k N_G(P_i)$ the system normalizer. Since $G = P_1P_2 \dots P_k$, any element $x\in G$ can be written as a product $x = x_1x_2 \dots x_k$, where $x_i\in P_i$. For $a\in G$ write
$$ [x, a] = [x_1x_2 \dots x_k, a] = [x_1, a]^{x_2 \dots x_k}[x_2, a]^{x_3 \dots x_k} \dots [x_k, a].$$

If $a\in T$, we have $[x_i, a]\in P_i$ and therefore $[x,a]$ is a product of at most $k$ elements from $X(G)$. So the index 
$[G: C_G([x,a])]$ is at most $m^k$. Obviously $k$ here is less than $m$. Thus, for any $x\in G$, $a\in T$ we have $|[x,a]^G|\leq m^k$, which is $m$-bounded. The combination of Lemma \ref{[G,T]} and Theorem \ref{shumyresult} now implies that $[G, T]' = G''$ has $m$-bounded order. 
\end{proof}

We say that a group is semisimple if it is a direct product of non-abelian simple groups. 

\begin{lemma} \label{semisimple} Let $G$ be a finite semisimple group such that $|x^G|\leq m$ whenever $x\in X(G)$. Then $G$ has $m$-bounded order.
\end{lemma}
\begin{proof} Write $G=S_1\times\dots\times S_k$, where $S_i$ are the simple factors. Each of these factors contain a subgroup of index at most $m$ whence, because of the simplicity, it follows  that $|S_i|\leq m!$. Therefore we only need to show that $k$ is $m$-bounded.

Recall that simple groups have even order \cite{ft}. For every $i=1,\dots,k$ choose a Sylow 2-subgroup $P_i$ of $S_i$ and elements $x_i\in P_i$ and $y_i\in S_i$ such that $1\neq[x_i,y_i]\in Z(P_i)$. Indeed, if the subgroup $P_i$ is non-abelian, both elements $x_i$ and $y_i$ can be chosen inside $P_i$. If $P_i$ is abelian, then the existence of $x_i$ and $y_i$ is immediate from the theorem of Frobenius \cite[Theorem 7.4.5]{go}. For every $i=1,\dots,k$ set $g_i=[x_i,y_i]$. Since the index of the centralizer of an element in a simple group is not a prime power \cite[Lemma 4.3.2]{go}, it follows that $|g_i^{S_i}|\geq15$. Let $g=\prod_ig_i$. Observe that $g=[\prod x_i,\prod y_i]\in\prod P_i$. We deduce that $g\in X(G)$ and $|g^{G}|\geq15^k$. Hence, $15^k\leq m$ and so $k\leq\log_{15}(m)$.
\end{proof}

Recall that the generalized Fitting subgroup $F^*(G)$ of a finite group $G$ is the product of the Fitting subgroup $F(G)$ and all subnormal quasisimple subgroups; here a group is quasisimple if it is perfect and its quotient by the centre is a nonabelian simple group. In any finite group $G$ we have $C_G(F^*(G))\leq F^*(G)$. Therefore if $|F^*(G)|=t$, then $|G|\leq t!$.

\begin{lemma} \label{index} Let $G$ be a finite group such that $|x^G|\leq m$ whenever $x\in X(G)$, and let $R=R(G)$ be the soluble radical of $G$. Then the quotient $G/R$ has $m$-bounded order.
\end{lemma} 
\begin{proof}  Without loss of generality we may assume that $R=1$. Then $F^*(G)$ is semisimple and by Lemma \ref{semisimple} the order of $F^*(G)$ is $m$-bounded. Since $|G|\leq|F^*(G)|!$, we immediately deduce that the order of $G$ is $m$-bounded. This proves the lemma.
\end{proof}

We are now ready to complete the proof of Theorem \ref{main}.

\begin{proof}[Proof of Theorem \ref{main}] Recall that $G$ is a finite group in which $|x^G|\leq m$ whenever $x$ is a commutator of prime power order. We wish to prove that $G''$ has $m$-bounded order. 

We will prove the stronger result that if $G$ is a finite group in which $|x^G|\leq m$ whenever $x\in X(G)$, then $G''$ has $m$-bounded order. 

Let $H=G^{(\infty)}$ be the last term of the derived series of $G$. Of course, $H=H'$. Note that $G/H$ is soluble and so by Lemma \ref{soluble} the second commutator subgroup of $G/H$ has $m$-bounded order. Therefore it is sufficient to show that $H$ has $m$-bounded order. Hence, without loss of generality, we assume that $G=H$, that is, $G=G'$.

Let $R=R(G)$ be the soluble radical of $G$. Lemma \ref{index} tells us that $R$ has $m$-bounded index in $G$. Therefore we can use induction on the index $[G:R]$.

If $[G:R]=1$, then $G$ is soluble. Since $G=G'$, it follows that $G=1$ and we are done. So we assume that $[G:R]\geq2$. Suppose $G$ contains a normal subgroup $M$ such that $R\leq M<G$. By induction $|M''|$ is $m$-bounded. Certainly, $M/M''\leq R(G/M'')$. Therefore the index of the soluble radical in $G/M''$ is strictly less than $[G:R]$ and the result follows by induction.

Hence, we can assume that such subgroups $M$ do not exist and so $G/R$ is simple.  Set $X=X(G)$ and observe that, since $G=G'$, it follows that $G$ is generated by $X$. By Lemma \ref{[G,x]} the order of the subgroup $[G,x]$ is $m$-bounded whenever $x\in X$.

Taking into account that $G/R$ is a nontrivial simple group, we conclude that there is $x_0\in X$ such that $[G,x_0]R=G$. Observe that the quotient $G/[G,x_0]$ is soluble. Since $G=G'$, we deduce that $G=[G,x_0]$. Hence, the order of $G$ is $m$-bounded, as required.
\end{proof}


\begin{thebibliography}{b}
\bibitem{as}  C. Acciarri, P. Shumyatsky, A stronger form of Neumann's BFC-theorem, Israel J. Math., {\bf 242}  (2021), 269--278.
\bibitem{dms}  E. Detomi, M. Morigi, P. Shumyatsky, BFC-theorems for higher commutator subgroups, Quart. J. Math., {\bf 70} (2019), 849--858.
\bibitem{dieshu} G. Dierings, P. Shumyatsky, Groups with boundedly finite conjugacy classes of commutators, Quart. J. Math., {\bf 69} (2018), 1047--1051.
\bibitem{dh} K. Doerk, T. Hawkes,  Finite Soluble Groups, de Gruyter, Berlin (1992).
\bibitem{ebeshu} S. Eberhard, P. Shumyatsky, Probabilistically nilpotent groups of class two. Math. Ann. {\bf 388} (2024), 1879--1902.
\bibitem{ft} W. Feit, J. G. Thompson, Solvability of groups of odd order, Pacific J. Math., {\bf 13} (1963), 773--1029.
\bibitem{go} D. Gorenstein, Finite Groups, Chelsea, New York (1980).
\bibitem{gura} R. M. Guralnick, A. Maroti, Average dimension of fixed point spaces with applications, J. Algebra {\bf  226} (2011), 298--308.
\bibitem{neumann} B. H. Neumann, Groups covered by permutable subsets, J. London Math. Soc., {\bf 29} (1954), 236–248.

\bibitem{neuvl} P. M. Neumann, M. R. Vaughan-Lee, An essay on BFC groups, Proc. Lond. Math. Soc., {\bf 35} (1977), 213–237.

\bibitem{robinson} D. J. S. Robinson, A course in the theory of groups, 2nd edn. Graduate Texts in Mathematics, 80. Springer-Verlag, New York, 1996.

\bibitem{sesha} D. Segal and A. Shalev, On groups with bounded conjugacy classes, Quart. J. Math. {\bf 50} (1999), 505--516.

\bibitem{shumy} P. Shumyatsky, Bounded Conjugacy Classes, Commutators and Approximate Subgroups, Quart. J. Math. {\bf 73} (2022),  679--684.

\bibitem{wie} J. Wiegold, Groups with boundedly finite classes of conjugate elements, Proc. Roy. Soc. London Ser. A, {\bf 238} (1957), 389--401.
\end{thebibliography}
\end{document}